\documentclass[11pt,a4paper]{article}

\usepackage{titlesec}
\usepackage{fancyhdr}
\usepackage{a4wide}
\usepackage{graphicx}
\usepackage{float}
\usepackage{amssymb}
\usepackage{amsmath}
\usepackage{amsthm}
\usepackage{color}
\usepackage{mathrsfs}
\usepackage{array}
\usepackage{eucal}
\usepackage{tikz}
\usepackage[T1]{fontenc}
\usepackage{inputenc}
\usepackage[english]{babel}
\usepackage{lmodern}
\usepackage{hyperref}
\usepackage{geometry}
\usepackage{changepage}
\changepage{0pt}{}{}{}{}{0pt}{}{0pt}{10pt}
\usepackage[numbers]{natbib}
\usepackage{hyperref}
\hypersetup{
pdfpagemode=none,
pdftoolbar=true,        
pdfmenubar=true,        
pdffitwindow=false,     
pdfstartview={Fit},    
pdftitle={Simple examples of perfectly invisible and trapped modes in  waveguides},    
pdfauthor={L. Chesnel, V. Pagneux},     
pdfsubject={},  
pdfcreator={L. Chesnel, V. Pagneux},   
pdfproducer={L. Chesnel, V. Pagneux}, 
pdfkeywords={}, 
pdfnewwindow=true,      
colorlinks=true,       
linkcolor=magenta,          
citecolor=red,        
filecolor=cyan,      
urlcolor=blue           
}

\newcommand{\dsp}{\displaystyle}

\newcommand{\om}{\omega}
\newcommand{\Om}{\Omega}
\newcommand{\mrm}[1]{\mathrm{#1}}

\newcommand{\Cplx}{\mathbb{C}}

\newcommand{\R}{\mathbb{R}}

\newcommand{\mH}{\mrm{H}}

\newcommand{\loc}{\mbox{\scriptsize loc}}

\newcommand{\rcoef}{\mathcal{R}}
\newcommand{\tcoef}{\mathcal{T}} 
\newcommand{\rN}{r}
\newcommand{\rD}{R}

\newtheorem{theorem}{Theorem}[section]
\newtheorem{lemma}{Lemma}[section]
\newtheorem{remark}{Remark}[section]

\newtheorem{proposition}{Proposition}[section]

\begin{document}

~\vspace{0.0cm}
\begin{center}
{\sc \bf\LARGE  Simple examples of perfectly invisible \\[6pt] and trapped modes in waveguides}
\end{center}

\begin{center}
\textsc{Lucas Chesnel}$^1$, \textsc{Vincent Pagneux}$^2$\\[16pt]
\begin{minipage}{0.96\textwidth}
{\small
$^1$ INRIA/Centre de math\'ematiques appliqu\'ees, \'Ecole Polytechnique, Universit\'e Paris-Saclay, Route de Saclay, 91128 Palaiseau, France;\\
$^2$ Laboratoire d'Acoustique de l'Universit\'e du Maine, Av. Olivier Messiaen, 72085 Le Mans, France.\\[10pt]
E-mails: \texttt{lucas.chesnel@inria.fr}, \texttt{vincent.pagneux@univ-lemans.fr}\\[-12pt]
\begin{center}
-- \today --
\end{center}
}
\end{minipage}
\end{center}
\vspace{0.4cm}

\noindent\textbf{Abstract.} We consider the propagation of waves in a waveguide with Neumann boundary conditions. We work at low wavenumber with only one propagating mode in the leads, all the other modes being evanescent. We assume that the waveguide is symmetric with respect to an axis orthogonal to the longitudinal direction and is endowed with a branch of height $L$ whose width coincides with the wavelength of the propagating modes. In this setting, tuning the parameter $L$, we prove the existence of simple geometries where the transmission coefficient is equal to one (perfect invisibility). We also show that these geometries, for possibly different values of $L$, support so called trapped modes (non zero solutions of finite energy of the homogeneous problem) associated with eigenvalues embedded in the continuous spectrum.\\

\noindent\textbf{Key words.} Waveguides, invisibility, trapped modes, scattering matrix, asymptotic analysis.

\section{Introduction}

In this work, we consider a problem of wave propagation, governed by the Helmholtz equation, in a waveguide unbounded in one direction, the longitudinal direction $(Ox)$, in frequency regime. We supplement it with Neumann boundary conditions. Such a problem appears for example in the theory of water-waves, in acoustics or in electromagnetism. We shall assume that the wavenumber $k$ is sufficiently small so that only one mode (the piston mode) propagates. We will be particularly interested in the scattering of the piston mode coming from $-\infty$ by the structure. To describe such a process, we introduce two complex coefficients, the so-called \textit{reflection} and \textit{transmission} coefficients, denoted $\rcoef$ and $\tcoef$, such that $\rcoef$ (resp. $\tcoef-1$) corresponds to the amplitude of the scattered field at $x=-\infty$ (resp. $x=+\infty$) (see (\ref{DefScatteringCoeff})). According to the conservation of energy, we have
\begin{equation}\label{NRJconservation}
|\rcoef|^2+|\tcoef|^2=1.
\end{equation}
In the first part of the article, we explain how to construct waveguides, different from the straight (reference) geometry, such that $\rcoef=0$, $\tcoef=1$. In this case, we shall say that the total field is a \textit{perfectly invisible mode}. The reason is that in this situation, the scattered field is exponentially decaying both at $\pm\infty$ and for an observer with measurement devices located far from the geometrical defect, everything happens like in the reference waveguide. In other words, the geometrical perturbation is invisible from far field measurements. The problem of imposing $\rcoef=0$, $\tcoef=1$ seems quite new in literature, at least when one looks for proofs. A simpler problem consists in finding \textit{non reflecting} geometries such that $\rcoef=0$ (and so $|\tcoef|=1$ according to (\ref{NRJconservation})). For such waveguides, all the energy is transmitted but there is a possible shift of phase for the field at $x=+\infty$. However, one can speak of backscattering invisibility. The latter problem has been investigated numerically ($|\rcoef|$ small) for example in \cite{PoNe14,EvMP14} for water wave problems and in \cite{AlSE08,EASE09,NgCH10,OuMP13,FuXC14}, with strategies based on the use of new ``zero-index'' and ``epsilon near zero'' metamaterials, in electromagnetism (see \cite{FlAl13} for an application to acoustics).\\
\newline
A perturbative approach, relying on the fact that $\rcoef=0$ in the straight geometry, has been proposed in \cite{BoNa13} to construct waveguides such that $\rcoef=0$, $|\tcoef|=1$ (see also \cite{BLMN15,ChNa16,BoCNSu,ChHS15} for  applications in related contexts). It has been adapted in \cite{BoCNSub} to solve the problem of imposing $\rcoef=0$, $\tcoef=1$. The technique, based on the proof of the implicit function theorem, allows one to design invisible perturbations which \textit{a priori} are small with respect to the wavelength. An alternative method has been developed in \cite{ChNPSu} to obtain larger invisible defects. In the latter work, it is explained how to get $\rcoef=0$, $|\tcoef|=1$ in waveguides that are symmetric with respect to the $(Oy)$ axis (perpendicular to the direction of propagation). The idea consists in using the symmetry properties and to play with a branch of finite height $L$ making $L\to+\infty$. In  \cite{ChNPSu}, it is also shown how to work with three branches to impose $\rcoef=0$, $\tcoef=1$. In the present article, we use a similar idea but, importantly, we work with only one branch whose width coincides with the wavelength of the incident wave. We provide examples of very simple geometries where $\rcoef=0$, $\tcoef=1$ (which is not obvious to attain in general).\\
\newline
In the second part of the paper, we will be concerned with so-called \textit{trapped modes}. We remind the reader that trapped modes are non zero solutions of the homogeneous problem (\ref{PbInitial}) (without source term) which are of finite energy. It has been known for a while that trapped modes play a key role in physical systems and they have been widely studied. In particular, literature concerning trapped modes is much more developed than that concerning invisible modes. We refer the reader for example to \cite{Urse51,Evan92,EvLV94,DaPa98,LiMc07,Naza10c,NaVi10,CPMP11,Pagn13}. More precisely, in our work, we will be interested in trapped modes associated with eigenvalues embedded in the continuous spectrum. Such trapped modes are also called \textit{Bound States in the Continuum} (BSCs or BICs) in quantum mechanics (see for example \cite{SaBR06,Mois09,GPRO10} as well as the recent review \cite{HZSJS16}). Such objects are difficult to observe. In particular they are unstable with respect to the geometry and a small perturbation transforms them in general into complex resonances \cite{AsPV00}. A classical method to construct trapped modes associated with eigenvalues embedded in the continuous spectrum consists in working with waveguides that are symmetric with respect to the $(Ox)$ axis \cite{Evan92,DaPa98,LiMc07,Pagn13}. This is interesting because it allows one to decouple symmetric and skew-symmetric modes. Then the idea is to design a defect such that trapped modes exist below the continuous spectrum for the problem with mixed boundary conditions satisfied by the skew-symmetric component of the field. Back to the original waveguide, this guarantees the existence of trapped modes. What we will do in the second part of the article is to propose a simple alternative mechanism to construct trapped modes associated with an eigenvalue embedded in the continuous spectrum. We emphasize that it will be based on symmetry arguments with respect to $(Oy)$ and not with respect to $(Ox)$. It will also involve the \textit{augmented scattering matrix}, a convenient tool introduced in \cite{NaPl94bis,KaNa02,Naza06,Naza11}, allowing one to detect the presence of trapped modes. \\
\newline
The outline is as follows. In Section \ref{SectionPerfectInvisibility}, we exhibit geometries admitting perfectly invisible modes. In Section \ref{SectionExistenceOfTrappedModes}, we provide examples of geometries supporting trapped modes. In Section \ref{SectionNumExpe}, we give numerical illustrations of the results showing also that other geometrical shapes can be considered. We end the paper with a short conclusion (Section \ref{SectionConclusion}) and an Annex where we give the proof of a classical result used in the analysis. The main results of this work are Theorem \ref{MainThmPart1} (existence of perfectly invisible modes) and Theorem \ref{MainThmPart2} (existence of trapped modes).

\section{Perfectly invisible modes}\label{SectionPerfectInvisibility}

\subsection{Setting}
~\vspace{-0.6cm}

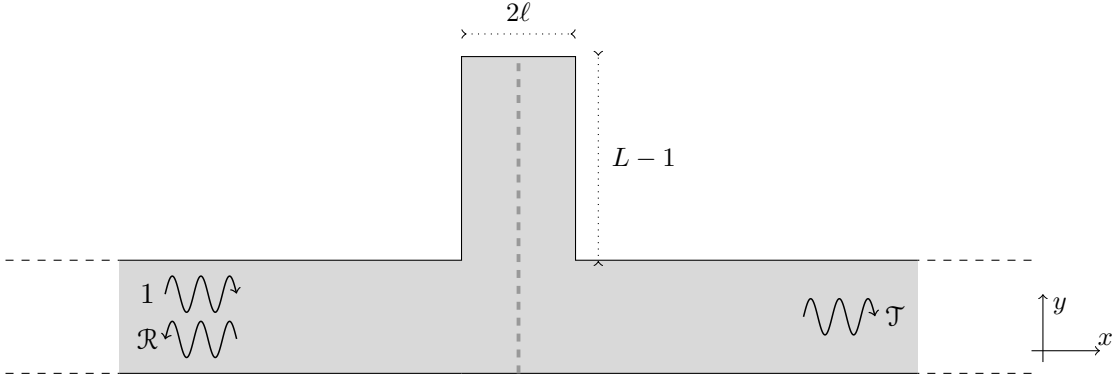
\begin{figure}[!ht]
\centering
\begin{tikzpicture}[scale=1.5]
\draw[fill=gray!30,draw=none](-2,1) rectangle (2,2);
\draw[fill=gray!30,draw=none](-5,1) rectangle (-2,2);
\draw[fill=gray!30,draw=none](-2,1.8) rectangle (-1,3.8);
\draw (-2,1)--(2,1); 
\draw (-5,2)--(-2,2)--(-2,3.8)--(-1,3.8)--(-1,2)--(2,2); 
\draw (-5,1)--(-2,1); 
\draw[dashed] (3,1)--(2,1); 
\draw[dashed] (3,2)--(2,2);
\draw[dashed] (-6,2)--(-5,2);
\draw[dashed] (-6,1)--(-5,1);
\draw[dotted,>-<] (-2,4)--(-1,4);
\draw[dotted,>-<] (-0.8,1.95)--(-0.8,3.85);
\node at (-1.5,4.2){\small $2\ell$};
\node at (-0.4,2.9){\small $L-1$};
\draw[dashed,line width=0.5mm,gray!80] (-1.5,1)--(-1.5,3.8);
\draw[->] (3,1.2)--(3.6,1.2);
\draw[->] (3.1,1.1)--(3.1,1.7);
\node at (3.65,1.3){\small $x$};
\node at (3.25,1.6){\small $y$};
\begin{scope}[xshift=-4.6cm,yshift=1.7cm,scale=0.8]
\draw[line width=0.2mm,->] plot[domain=0:pi/4,samples=100] (\x,{0.2*sin(20*\x r)}) node[anchor=west] {\hspace{-2.4cm}$1$};
\end{scope}
\begin{scope}[xshift=-4.6cm,yshift=1.3cm,scale=0.8]
\draw[line width=0.2mm,<-] plot[domain=0:pi/4,samples=100] (\x,{0.2*sin(20*\x r)}) node[anchor=west] {$\hspace{-2.4cm}\mathcal{R}$};
\end{scope}
\begin{scope}[xshift=1cm,yshift=1.5cm,scale=0.8]
\draw[line width=0.2mm,->] plot[domain=0:pi/4,samples=100] (\x,{0.2*sin(20*\x r)}) node[anchor=west] {$\mathcal{T}$};
\end{scope}
\end{tikzpicture}
\caption{Geometry of $\Om_L$. \label{DomainOriginal}} 
\end{figure}

\noindent Pick a wavenumber $k\in(0;\pi)$. Set $\ell=\pi/k$ and for $L>1$, define the waveguide (see Figure \ref{DomainOriginal})
\begin{equation}\label{defOriginalDomain}
\Om_L:=\{ (x,y)\in\R\times(0;1)\ \cup\  (-\ell;\ell)\times [1;L)\}. 
\end{equation}
The value $2\pi/k$ for the width of the vertical branch of $\Om_L$ is very important in the analysis we develop below as we will see later. For the main exposition, we will stick with this simple waveguide $\Om_L$. However the method we propose works in other geometries, such as the one described in (\ref{defDomainBis}) (see Figure \ref{figResultT1Gros})). We consider the Helmholtz problem with Neumann boundary conditions (sound hard walls in acoustics)
\begin{equation}\label{PbInitial}
\begin{array}{|rcll}
\Delta v + k^2 v & = & 0 & \mbox{ in }\Om_L\\[3pt]
 \partial_nv  & = & 0  & \mbox{ on }\partial\Om_L.
\end{array}
\end{equation}
In (\ref{PbInitial}), $\Delta$ denotes the $\mrm{2D}$ Laplace operator and $n$ refers to the outer unit normal vector to $\partial\Om_{L}$. Set
\[
w^{\pm}_{1}(x,y)=\cfrac{1}{\sqrt{2k}}\,e^{\mp i k x}\quad\mbox{ and }\quad  \hat{w}^{\pm}_{1}(x,y)=\cfrac{1}{\sqrt{2k}}\,e^{\pm i k x}.
\]

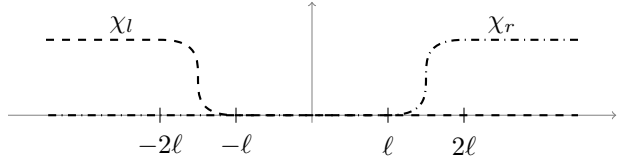
\begin{figure}[!ht]
\centering
\begin{tikzpicture}[scale=1]
\draw[gray,very thin,->] (-4,0)--(4,0);
\draw[gray,very thin,->] (0,-0.1)--(0,1.5);
\draw[dashed,thick] (-3.5,1)--(-2,1) .. controls (-1,1) and (-2,0) .. (-1,0) -- (3.5,0);
\draw[dash dot,thick] (3.5,1)--(2,1) .. controls (1,1) and (2,0) .. (1,0) -- (-3.5,0);
\draw (-2,-0.1)--(-2,0.1);
\draw (-1,-0.1)--(-1,0.1);
\draw (2,-0.1)--(2,0.1);
\draw (1,-0.1)--(1,0.1);
\node at (-2,-0.4){\small $-2\ell$};
\node at (-1,-0.4){\small $-\ell$};
\node at (2,-0.4){\small $2\ell$};
\node at (1,-0.4){\small $\ell$};
\node at (-2.5,1.2){\small $\chi_l$};
\node at (2.5,1.2){\small $\chi_r$};
\end{tikzpicture}
\caption{Graphs of the cut-off functions $\chi_l$, $\chi_r$.\label{CutOffFunctions}} 
\end{figure}

\noindent Let $\chi_l\in\mathscr{C}^{\infty}(\R^2)$ (resp. $\chi_r\in\mathscr{C}^{\infty}(\R^2)$) be a cut-off function equal to one for $x\le -2\ell$ (resp. $x\ge2\ell$) and to zero for $x\ge -\ell$ (resp. $x\le \ell$) (see Figure \ref{CutOffFunctions}). In order to describe the scattering process of the incident piston modes $w_1^{-}$ coming from $x=-\infty$ (we assume a time dependence in $e^{-i\om t}$ where $\om$ is the angular frequency) by the structure, we introduce the following solution of  Problem (\ref{PbInitial}) 
\begin{equation}\label{DefScatteringCoeff}
v= \chi_l\,(w^{-}_{1}+\rcoef\,w^{+}_{1})+\chi_r\,\tcoef\,\hat{w}^{+}_{1}+\tilde{v},
\end{equation}
where $\rcoef,\,\tcoef\in\Cplx$ and where $\tilde{v}$ decays exponentially as $O(e^{-\sqrt{\pi^2-k^2}|x|})$ for $x\to\pm\infty$. The \textit{reflection coefficient} $\rcoef$ and \textit{transmission coefficient} $\tcoef$ in (\ref{DefScatteringCoeff}) are uniquely defined. Note that the cut-off functions $\chi_l$, $\chi_r$ in (\ref{DefScatteringCoeff}) are just a convenient way to write radiation conditions at $x=\pm\infty$. Physically, the function $v$ defined in (\ref{DefScatteringCoeff}) corresponds to the so called total field associated to the incident piston wave propagating from $-\infty$ to $+\infty$. According to the energy conservation, we have 
\begin{equation}\label{ConservationOfNRJ}
|\rcoef|^2+|\tcoef|^2=1.
\end{equation}
The coefficients $\rcoef$ and $\tcoef$ depend on the features of the geometry, in particular on $L$. From time to time, we shall write $\rcoef(L)$, $\tcoef(L)$ instead of $\rcoef$, $\tcoef$. In this section, we show that, there are some $L>1$ such that $\rcoef=0$, $\tcoef=1$ (\textit{perfect invisibility}). To obtain such particular values for the scattering coefficients, we will use the fact that the geometry is symmetric with respect to the $(Oy)$ axis. 
\subsection{Decomposition in two half-waveguide problems}
Define the half-waveguide 
\begin{equation}\label{defHalfWaveguide}
\om_L:=\{(x,y)\in\Om_L\,|\,x<0\}
\end{equation}
(see Figure \ref{LimitDomain} left). Introduce the problem with Neumann boundary conditions 
\begin{equation}\label{PbChampTotalSym}
\begin{array}{|rcll}
\Delta u +k^2 u & = & 0 & \mbox{ in }\om_L\\[3pt]
 \partial_nu  & = & 0  & \mbox{ on }\partial\om_L
\end{array}
\end{equation}
as well as the problem with mixed boundary conditions 
\begin{equation}\label{PbChampTotalAntiSym}
\begin{array}{|rcll}
\Delta U + k^2 U & = & 0 & \mbox{ in }\om_L\\[3pt]
 \partial_nU  & = & 0  & \mbox{ on }\partial\om_L\cap\partial\Om_L \\[3pt]
U  & = & 0  & \mbox{ on }\Sigma_L:=\{0\}\times(0;L).
\end{array}
\end{equation}
Problems (\ref{PbChampTotalSym}) and (\ref{PbChampTotalAntiSym}) admit respectively the solutions 
\begin{eqnarray}
\label{defZetaLsym} u &=& w^+_1+\rN\,w^-_1 + \tilde{u},\qquad\hspace{2mm}\mbox{ with }\tilde{u}\in\mH^1(\om_L),\\[3pt]
\label{defZetaLanti}U &=& w^+_1+\rD\,w^-_1 + \tilde{U},\qquad\mbox{ with }\tilde{U}\in\mH^1(\om_L),
\end{eqnarray} 
where $\rN$, $\rD\in\Cplx$ are uniquely defined. Moreover, due to conservation of energy, one has
\begin{equation}\label{NRJHalfguide}
|\rN|=|\rD|=1.
\end{equation}
Direct inspection shows that if $v$ is a solution of Problem (\ref{PbInitial}) then we have $v(x,y)=(u(x,y)+U(x,y))/2$ in $\om_L$ and $v(x,y)=(u(-x,y)-U(-x,y))/2$ in $\Om_L\setminus\overline{\om_L}$ (up possibly to a term which is exponentially decaying at $\pm\infty$ if there is a trapped mode at the given wavenumber $k$). We deduce that the scattering coefficients $\rcoef$, $\tcoef$ appearing in the decomposition (\ref{DefScatteringCoeff}) of $v$ are such that
\begin{equation}\label{Formulas}
\rcoef=\frac{\rN+\rD}{2}\qquad\mbox{ and }\qquad \tcoef=\frac{\rN-\rD}{2}.
\end{equation}
\subsection{Consequence of the choice $\ell=\pi/k$}
In the particular geometry considered here, one observes that $u':=w^+_1+w^-_1=\sqrt{2/k}\cos(kx)$ is a solution to (\ref{PbChampTotalSym}) for all $L>1$. Note that this property is based on the fact that $\ell=\pi/k$, the width of the vertical branch of $\om_L$, coincides with the half wavelength of the waves $w^{\pm}_1$. By uniqueness of the definition of the coefficient $\rN$ in (\ref{defZetaLsym}), we deduce that 
\begin{equation}\label{PartRelation}
\rN=\rN(L)=1\qquad\mbox{ for all $L>1$}. 
\end{equation}
From this important remark, we see that to get $\tcoef=1$, it remains to find $L$ such that $R=R(L)=-1$. This is the goal of the next section.

\subsection{Asymptotic behaviour of the reflection coefficient for the problem with mixed boundary conditions}\label{paragraphAsymptoAnalysis}
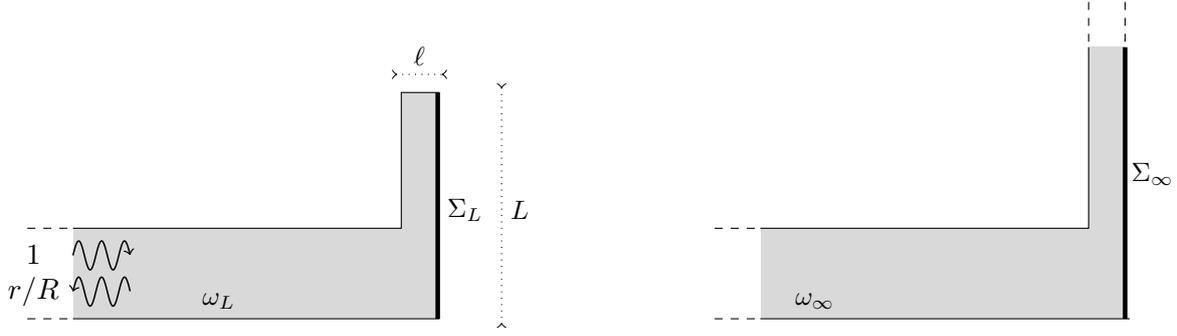
\begin{figure}[!ht]
\centering
\begin{tikzpicture}[scale=1.2]
\draw[fill=gray!30,draw=none](-4,0) rectangle (0,1);
\draw[fill=gray!30,draw=none](-0.4,0) rectangle (0,2.5);
\draw (-4,1)--(-0.4,1)--(-0.4,2.5)--(0.02,2.5);
\draw (0.02,0)--(-4,0);
\draw[line width=0.6mm] (0,0)--(0,2.5);
\draw[dashed] (-4.5,1)--(-4,1); 
\draw[dashed] (-4.5,0)--(-4,0); 
\node at (0.3,1.2){\small $\Sigma_{L}$};
\node at (-2.4,0.2){\small $\om_{L}$};
\draw[dotted,>-<] (-0.5,2.7)--(0.1,2.7);
\draw[dotted,>-<] (0.7,-0.1)--(0.7,2.6);
\node at (0.9,1.2){\small $L$};
\node at (-0.2,2.9){\small $\ell$};
\phantom{\draw[dashed] (0,3.5)--(0,3); }
\begin{scope}[xshift=-4cm,yshift=0.7cm,scale=0.8]
\draw[line width=0.2mm,->] plot[domain=0:pi/4,samples=100] (\x,{0.2*sin(20*\x r)}) node[anchor=west] {\hspace{-2.6cm}$1$};
\end{scope}
\begin{scope}[xshift=-4cm,yshift=0.3cm,scale=0.8]
\draw[line width=0.2mm,<-] plot[domain=0:pi/4,samples=100] (\x,{0.2*sin(20*\x r)}) node[anchor=west] {$\hspace{-2.6cm}r/R$};
\end{scope}
\end{tikzpicture}\qquad\qquad\qquad\begin{tikzpicture}[scale=1.2]
\draw[fill=gray!30,draw=none](-4,0) rectangle (0,1);
\draw[fill=gray!30,draw=none](-0.4,0) rectangle (0,3);
\draw (-4,0)--(0.05,0); 
\draw[line width=0.6mm] (0,0)--(0,3);
\draw (-4,1)--(-0.4,1)--(-0.4,3);
\draw[dashed] (-4.5,1)--(-4,1); 
\draw[dashed] (-4.5,0)--(-4,0); 
\draw[dashed] (0,3.5)--(0,3); 
\draw[dashed] (-0.4,3.5)--(-0.4,3); 
\node at (0.3,1.6){\small $\Sigma_{\infty}$};
\node at (-3.4,0.2){\small $\om_{\infty}$};
\phantom{\draw[dotted,>-<] (0.3,-0.1)--(0.3,2.6);}
\end{tikzpicture}
\caption{Domains $\om_{L}$ (left) and $\om_{\infty}$ (right).\label{LimitDomain}} 
\end{figure}
\noindent In this paragraph, we study the behaviour of the reflection coefficient $\rD=\rD(L)$ of the half-waveguide problem with mixed boundary conditions (\ref{PbChampTotalAntiSym})  as $L\to\infty$. We follow the approach proposed in \cite{ChNPSu}. In the analysis, the properties of Problem (\ref{PbChampTotalAntiSym}) set in the limit geometry $\om_{\infty}$ (see Figure \ref{LimitDomain} right) obtained formally taking $L\to+\infty$, play a key role. Denote
\begin{equation}\label{DefModeVert}
w^{\pm}_{2}(x,y)=\cfrac{1}{\sqrt{\gamma\ell}}\,e^{\pm i \gamma y}\sin( \frac{\pi x}{2\ell}),\qquad \mbox{with }\gamma:=\sqrt{k^2-(\pi/(2\ell))^2}=k\sqrt{3}/2,
\end{equation}
the modes propagating in the vertical branch of $\om_{\infty}$. In $\om_{\infty}$, there are the solutions 
\begin{equation}\label{defMatrixScaLim}
\begin{array}{lcl}
U_{1}^{\infty}&=& \chi_l(w^-_{1}+S^{\infty}_{11}\,w^+_{1})+\chi_t\,S^{\infty}_{12}\,w^+_{2}+\tilde{U}_{1}^{\infty}\\[4pt]
U_{2}^{\infty}&=&\chi_l\,S^{\infty}_{21}\,w^+_{1}+\chi_t\,(w^-_{2}+S^{\infty}_{22}\,w^+_{2})+\tilde{U}_{2}^{\infty},
\end{array}
\end{equation}
where $\tilde{U}_{1}^{\infty}$, $\tilde{U}_{2}^{\infty}$ decay exponentially at infinity. Here $\chi_t\in\mathscr{C}^{\infty}(\R^2)$ is a cut-off function equal to one for $y\ge 1+2\tau$ and to zero for $y\le 1+\tau$, where $\tau>0$ is a given constant. The scattering matrix
\begin{equation}\label{UnboundedScatteringMatrix}
\mathcal{S}^{\infty}:=\left(\begin{array}{cc}
S^{\infty}_{11} & S^{\infty}_{12} \\
S^{\infty}_{21} & S^{\infty}_{22} \\
\end{array}\right)\in\mathbb{C}_{2\times 2}
\end{equation}
is uniquely defined. Moreover, as shown in Proposition \ref{PropositionUnitary} in Annex (this is a classical result), $\mathcal{S}^{\infty}$ is unitary ($\mathcal{S}^{\infty}\overline{\mathcal{S}^{\infty}}^{\top}=\mrm{Id}_{2\times2}$) and symmetric. For the function 
$U$ defined in (\ref{defZetaLanti}), following for example \cite[Chap. 5, \S5.6]{MaNP00}, we make the ansatz
\begin{equation}\label{DefAnsatzs}
U = U_{1}^{\infty}+A(L)\,U_{2}^{\infty}+\dots 
\end{equation}
where $A(L)$ is a gauge function to determine and where the dots correspond to a small remainder. On $(-\ell;0)\times\{L\}$, the condition $\partial_n U=0$ leads to choose $A(L)$ such that 
\[
S^{\infty}_{12}\,e^{i\gamma L}+A(L)\,(-e^{-i\gamma L}+S^{\infty}_{22}\,e^{i\gamma L})=0\qquad\Leftrightarrow\qquad A(L)=\dsp\cfrac{S^{\infty}_{12}}{e^{-2i\gamma L}-S^{\infty}_{22}}\ . 
\]
We shall consider ansatz (\ref{DefAnsatzs}) when $|S^{\infty}_{22}|\ne1$. Since $\mathcal{S}^{\infty}$ is unitary and symmetric, this is equivalent to assume that $S^{\infty}_{12}=S^{\infty}_{21}\ne 0$. This assumption is needed so that a coupling exists between the two channels of $\om_{\infty}$. In this case, the gauge function $A(L)$ is well-defined for all $L>1$. If $|S^{\infty}_{22}|=1 \Leftrightarrow S^{\infty}_{12}=0$, we can show that as $L\to+\infty$, the reflection transmission $\rD(L)$ defined in (\ref{defZetaLanti}) tends to $S^{\infty}_{11}$. This exceptional case is not interesting for our analysis and therefore, we discard it assuming that $k\in(0;\pi)$ is such that $|S^{\infty}_{22}|\ne1 \Leftrightarrow S^{\infty}_{12}\ne0$ (we shall see in \S\ref{paragraphPerfectRef} below an example of situation where numerically $S_{12}^{\infty}\ne0$). Then we can prove that $\rD(L)=\rD^{\mrm{asy}}(L)+\dots$, with
\begin{equation}\label{defCoeffLim}
\rD^{\mrm{asy}}(L)=S^{\infty}_{11}+A(L)\,S^{\infty}_{21}=S^{\infty}_{11}+\dsp\cfrac{S^{\infty}_{12}\,S^{\infty}_{21}}{e^{-2i\gamma L}-S^{\infty}_{22}}\ .
\end{equation}
Here the dots stand for exponentially small terms. More precisely, we can establish (work as in \cite[Chap. 5, \S5.6]{MaNP00}) an error estimate of the form $|\rD(L)-\rD^{\mrm{asy}}(L)| \le C\,e^{-\eta L}$ where $C$ is a constant independent of $L>1$ and $\eta:=\sqrt{9\pi^2/(4\ell^2)-k^2}/2$. Observe that $L\mapsto \rD^{\mrm{asy}}(L)$ is a periodic function of period $\pi/\gamma=2\pi/(k\sqrt{3})=2\ell/\sqrt{3}$. As a consequence, $L\mapsto \rD(L)$ is ``almost periodic'' of period $2\pi/(k\sqrt{3})$. Denote $\mathscr{C}:=\{z\in\Cplx\,|\,|z|=1\}$ the unit circle. As $L$ tends to $+\infty$, the coefficient $\rD^{\mrm{asy}}(L)$ runs on the set
\begin{equation}\label{setSPlusMoins}
\{S^{\infty}_{11}+\dsp\frac{S^{\infty}_{12}\,S^{\infty}_{21}}{z-S^{\infty}_{22}} \,|\,z\in\mathscr{C}\}.
\end{equation}
Using classical results concerning the M\"{o}bius transform (see e.g.  \cite[Chap. 5]{Henr74}), one finds that this set is the circle centered at
\begin{equation}\label{eqnCenters}
\alpha_{\rD}:=S^{\infty}_{11}+\dsp\frac{S^{\infty}_{12}\,\overline{S^{\infty}_{22}}\,S^{\infty}_{21}}{1-|S^{\infty}_{22}|^{2}}\quad\mbox{ of radius }\quad \rho_{\rD}:=\dsp\frac{|S^{\infty}_{12}\,S^{\infty}_{21}|}{1-|S^{\infty}_{22}|^{2}}.
\end{equation}
\begin{proposition}\label{PropositionGoesThroughZero}
Assume that $S_{12}^{\infty}\ne0$. Then the set (\ref{setSPlusMoins}) is nothing else but the unit circle $\mathscr{C}$. 
\end{proposition}
\begin{proof}
Since $\mathcal{S}^{\infty}$ is unitary, we have $|S^{\infty}_{12}|^2+|S^{\infty}_{22}|^2=1$ and $S^{\infty}_{11}\overline{S^{\infty}_{12}}+S^{\infty}_{12}\overline{S^{\infty}_{22}}=0$. We deduce that $\alpha_{\rD}=0$ and $\rho_{\rD}=1$. 
\end{proof}
\noindent Proposition \ref{PropositionGoesThroughZero} together with the error estimate $|\rD(L)-\rD^{\mrm{asy}}(L)| \le C\,e^{-\eta L}$ and the energy conservation relation $|\rD(L)|=1$ show that $L\mapsto\rD(L)$ runs on the unit circle as $L\to+\infty$. As a consequence, the reflection coefficient for the problem with mixed boundary conditions $L\mapsto\rD(L)$ passes (exactly) through the point of affix $-1+0i$ an infinite number of times.

\subsection{Existence of perfectly invisible modes}\label{paragraphConcluPart1}

In Formula (\ref{Formulas}), we found that the transmission coefficient $\tcoef$ in the full waveguide $\Om_L$ satisfies $\tcoef=(\rN-\rD)/2$. Since $\rN=1$ (Formula (\ref{PartRelation})), from the analysis of the previous paragraph for the map $L\mapsto R(L)$, we deduce that $L\mapsto \tcoef(L)$ runs continuously and almost periodically on the circle of radius $1/2$ centered at $1/2+0i$ in the complex plane. In particular, $L\mapsto \tcoef(L)$ passes through the point of affix $1+0i$ almost periodically. The period is $2\pi/(k\sqrt{3})$. We summarize this result in the following theorem, the main result of the section. 
\begin{theorem}\label{MainThmPart1}
Assume that the coefficient $S_{12}^{\infty}$ in (\ref{defMatrixScaLim}) satisfies $S_{12}^{\infty}\ne0$. Then the complex curve $L\mapsto \tcoef(L)$ for the transmission coefficient passes through the point of affix $1+0i$ an infinite number of times as $L\to+\infty$.
\end{theorem}
\begin{remark}
Again, we mention that in \S\ref{paragraphPerfectRef} below, we provide an example of situation where numerically $S_{12}^{\infty}\ne0$. 
\end{remark}
\begin{remark}\label{RmkMirorEffect}
The analysis above also shows that there are some $L>1$ such that $\tcoef(L)=0$ (\textit{perfect reflection}) and $\rcoef(L)=1$ (see an illustration with Figure \ref{MirrorEffect} below). In that case, all the energy sent in the waveguide is backscattered at $x=-\infty$. But this is not surprising because it has been proved in \cite{ChPaSu} that this mirror effect appears naturally, even in waveguides which are not symmetric with respect to the $(Oy)$ axis and for almost all widths of the vertical branch of $\Om_L$. 
\end{remark}
\section{Trapped modes}\label{SectionExistenceOfTrappedModes}

In this section, we prove for that certain values of $L>1$, trapped modes exist for the half-waveguide problem with Neumann boundary conditions defined in (\ref{PbChampTotalSym}). We remind the reader that we say that $u$ is a trapped mode for Problem (\ref{PbChampTotalSym}) if $u$ belongs to the Sobolev space $\mH^1(\om_L)$ and satisfies (\ref{PbChampTotalSym}). Using a symmetry argument with respect to the line $\{x=0\}$, this will prove the existence of trapped modes for the Neumann problem set in the original domain $\Om_L$ defined in (\ref{defOriginalDomain}).

\subsection{Setting }

We shall use the same notation as in the previous section. Additionally, set $\beta=\sqrt{\pi^2-k^2}$ and define 
\[
W^{\pm}_{2}(x,y)=\cfrac{1}{\sqrt{2\beta}}\,(e^{-\beta x}\mp i e^{\beta x })\cos(\pi y).
\]
Note the particular definition of the functions $W^{\pm}_{2}$ which are ``wave packets'', combinations of exponentially decaying and growing modes as $x\to-\infty$. The normalisation coefficient for $W^{\pm}_{2}$ is chosen so that the matrix defined in (\ref{AugmentedScatteringDef}) is unitary. In \cite{NaPl94bis,KaNa02,Naza06,Naza11}, it is proved that in the half-waveguide $\om_L$ (unbounded in the left direction), there are the solutions
\begin{equation}\label{defu1u2}
\begin{array}{lcl}
u_{1}&=&w^-_{1}+s_{11}\,w^+_{1}+s_{12}\,W^+_{2}+\tilde{u}_{1}\\[4pt]
u_{2}&=&W^-_{2}+s_{21}\,w^+_{1}+s_{22}\,W^+_{2}+\tilde{u}_{2}
\end{array}
\end{equation}
where $\tilde{u}_{1}$, $\tilde{u}_{2}$ decay as $O(e^{\sqrt{4\pi^2-k^2}x})$ when $x\to-\infty$. The complex constants $s_{ij}$, $i,j\in\{1,2\}$ in (\ref{defu1u2}) are uniquely defined. They allow us to define the augmented scattering matrix introduced in \cite{NaPl94bis,KaNa02,Naza06,Naza11}
\begin{equation}\label{AugmentedScatteringDef}
\mathbb{S}:=\left(\begin{array}{cc}
s_{11} & s_{12}\\
s_{21} & s_{22}
\end{array}\right)\in\mathbb{C}_{2\times 2}.
\end{equation}
Working exactly as in the proof of Proposition \ref{PropositionUnitary} in Annex, one shows that the matrix $\mathbb{S}$ is unitary ($\mathbb{S}\,\overline{\mathbb{S}}^{\top}=\mrm{Id}_{2\times2}$) and symmetric ($s_{21}=s_{12}$). This augmented scattering matrix turns out be a very efficient tool to detect the presence of trapped modes. Indeed, we have the following algebraic criterion (see e.g. \cite[Thm. 2]{Naza11}).
\begin{lemma}\label{LemmaExistenceTrappedMode}
If $s_{22}=-1$, then $u_{2}$ is a trapped mode for Problem (\ref{PbChampTotalSym}) set in $\om_L$. 
\end{lemma}
\begin{remark}
Note that $s_{22}=-1$ is only a sufficient criterion of existence of trapped modes. Indeed the geometry $\om_L$ can support trapped modes for Problem (\ref{PbChampTotalSym}) with $s_{22}\ne-1$. In this case, this trapped mode must decay as $O(e^{\sqrt{4\pi^2-k^2}x})$ when $x\to-\infty$.
\end{remark}
\begin{proof}
If $s_{22}=-1$, since $\mathbb{S}$ is unitary, then $s_{21}=0$. In such a situation, according to (\ref{defu1u2}), we have $u_2=-i\sqrt{2/\beta}\,e^{\beta x}\cos(\pi y)+O(e^{\sqrt{4\pi^2-k^2}x})$ as $x\to-\infty$. This shows that $u_2\not\equiv0$ belongs to $\mH^1(\om_L)$. In other words $u_2$ is a trapped mode. 
\end{proof}

\subsection{Asymptotic behaviour of the augmented scattering matrix and existence of trapped modes}

In this paragraph, we study the asymptotic behaviour of $\mathbb{S}=\mathbb{S}(L)$ as $L\to+\infty$.\\
\newline
$\star$ As in the previous section, we first observe that $u'=w^+_1+w^-_1=\sqrt{2/k}\cos(kx)$ is a solution to (\ref{PbChampTotalSym}) for all $L>1$. From the uniqueness of the definition of the coefficients $s_{ij}$ in (\ref{defu1u2}), we deduce that for all $L>1$, we have $s_{11}(L)=1$ and $s_{12}(L)=s_{21}(L)=0$. Since $\mathbb{S}$ is unitary, we infer that $|s_{22}(L)|=1$ for all $L>1$.\\
\newline 
$\star$ It remains to investigate the behaviour of $s_{22}=s_{22}(L)$ as $L\to+\infty$. We adapt a bit what has been done in \S\ref{paragraphAsymptoAnalysis}. In the vertical branch of $\om_{\infty}$, the limit geometry of $\om_L$ as $L\to+\infty$, the following functions 
\[
w^{\pm}_{3}(x,y)=\cfrac{1}{\sqrt{2k\ell}}\,e^{\pm i k y},\qquad w^{\pm}_{4}(x,y)=\cfrac{1}{\sqrt{\ell}}\,(y\mp i )\cos(\pi \frac{x}{\ell})
\]
are propagating modes for Problem (\ref{PbChampTotalSym}). Note that the width $\ell>0$ has been chosen so that the frequency $k=\pi/\ell$ is a threshold frequency for Problem (\ref{PbChampTotalSym}) in the vertical branch of $\om_{\infty}$. This explains the special form of the modes $w_4^{\pm}$. The normalisation coefficients for $w^{\pm}_{3}$, $w^{\pm}_{4}$ are chosen so that the matrix $\mathbb{S}^{\infty}$ defined in (\ref{DefAugScaMat}) is unitary. In $\om_{\infty}$, there are the solutions 
\[
\begin{array}{lcl}
u_{1}^{\infty}&=& \chi_l(w^-_{1}+s^{\infty}_{11}\,w^+_{1}+s^{\infty}_{12}\,W^+_{2})+\chi_t\,(s^{\infty}_{13}\,w^+_{3}+s^{\infty}_{14}\,w^+_{4})+\tilde{u}_{1}^{\infty}\\[4pt]
u_{2}^{\infty}&=& \chi_l(W^-_{2}+s^{\infty}_{21}\,w^+_{1}+s^{\infty}_{22}\,W^+_{2})+\chi_t\,(s^{\infty}_{23}\,w^+_{3}+s^{\infty}_{24}\,w^+_{4})+\tilde{u}_{2}^{\infty}\\[4pt]
u_{3}^{\infty}&=& \chi_l(s^{\infty}_{31}\,w^+_{1}+s^{\infty}_{32}\,W^+_{2})+\chi_t\,(w^-_{3}+s^{\infty}_{33}\,w^+_{3}+s^{\infty}_{34}\,w^+_{4})+\tilde{u}_{3}^{\infty}\\[4pt]
u_{4}^{\infty}&=& \chi_l(s^{\infty}_{41}\,w^+_{1}+s^{\infty}_{42}\,W^+_{2})+\chi_t\,(w^-_{4}+s^{\infty}_{43}\,w^+_{3}+s^{\infty}_{44}\,w^+_{4})+\tilde{u}_{4}^{\infty},
\end{array}
\]
where $\tilde{u}_{1}^{\infty}$, $\tilde{u}_{2}^{\infty}$, $\tilde{u}_{3}^{\infty}$, $\tilde{u}_{4}^{\infty}$ decay as $O(e^{\sqrt{4\pi^2-k^2}x})$ for $x\to-\infty$ and as $O(e^{-\sqrt{4\pi^2/\ell^2-k^2}y})$ for $y\to+\infty$. The augmented scattering matrix
\begin{equation}\label{DefAugScaMat}
\mathbb{S}^{\infty}:=\left(\begin{array}{cccc}
s^{\infty}_{11} & s^{\infty}_{12} & s^{\infty}_{13}& s^{\infty}_{14}\\
s^{\infty}_{21} & s^{\infty}_{22} & s^{\infty}_{23}& s^{\infty}_{24} \\
s^{\infty}_{31} & s^{\infty}_{32} & s^{\infty}_{33}& s^{\infty}_{34}\\
s^{\infty}_{41} & s^{\infty}_{42} & s^{\infty}_{43}& s^{\infty}_{44}
\end{array}\right)\in\mathbb{C}_{4\times 4}
\end{equation}
is uniquely defined, unitary ($\mathbb{S}^{\infty}\overline{\mathbb{S}^{\infty}}^{\top}=\mrm{Id}_{4\times4}$) and symmetric (again, work as in the proof of Proposition \ref{PropositionUnitary} in Annex to prove the two latter properties). For $u_{2}$, we make the ansatz (see \cite[Chap. 5, \S5.6]{MaNP00})
\begin{equation}\label{ansatzTrapped}
\begin{array}{lcl}
u_{2} &=& u_{2}^{\infty}+a(L)\,u_{3}^{\infty}+b(L)\,u_{4}^{\infty}+\dots \  ,
\end{array}
\end{equation}
where $a(L)$, $b(L)$ are gauge functions to determine and where the dots stand for small remainders. On $(-\ell;0)\times\{L\}$, the condition $\partial_nu_{2}=0$ leads to choose $a(L)$, $b(L)$ such that 
\begin{equation}\label{system}
\begin{array}{c}
s^{\infty}_{23}\,e^{ikL}+a(L)\,(-e^{-ikL}+s^{\infty}_{33}\,e^{ikL})+b(L)\,s^{\infty}_{43}\,e^{ikL}=0\\[4pt]
s^{\infty}_{24}+a(L)\,s^{\infty}_{34}+b(L)(1+s^{\infty}_{44})=0.
\end{array}
\end{equation}
This yields 
\[
a(L)=\cfrac{s^{\infty}_{24}s^{\infty}_{43}-s^{\infty}_{23}(1+s^{\infty}_{44})}{(1+s^{\infty}_{44})(-e^{-2ikL}+s^{\infty}_{33})-s^{\infty}_{34}s^{\infty}_{43}}\quad\mbox{ and }\quad b(L)=\cfrac{s^{\infty}_{24}(e^{-2ikL}-s^{\infty}_{33})+s^{\infty}_{23}s^{\infty}_{34}}{(1+s^{\infty}_{44})(-e^{-2ikL}+s^{\infty}_{33})-s^{\infty}_{34}s^{\infty}_{43}}.
\]
We shall consider ansatz (\ref{ansatzTrapped}) for $u_2$ when $s^{\infty}_{44}\ne-1$. If $s^{\infty}_{44}=-1$, according to relations (\ref{relPart}) below, then $s^{\infty}_{11}=1$. Since $\mathbb{S}^{\infty}$ is unitary and symmetric, we deduce that $s^{\infty}_{12}=s^{\infty}_{13}=s^{\infty}_{14}=s^{\infty}_{24}=s^{\infty}_{34}=s^{\infty}_{21}=s^{\infty}_{31}=s^{\infty}_{41}=s^{\infty}_{42}=s^{\infty}_{43}=0$. In such a situation, an analysis similar to what has been done in \S\ref{paragraphAsymptoAnalysis} can be developed when $|s^{\infty}_{33}|\ne 1$ allowing us to show the existence of trapped modes for certain $L>1$. We will not consider this rather exceptional case in the following. Instead, we will assume that $k\in(0;\pi)$ is such that
\begin{equation}\label{AssumptionsCoef}
s^{\infty}_{14}\ne0\qquad\mbox{ and }\qquad |s^{\infty}_{33}-\cfrac{s^{\infty}_{13}s^{\infty}_{34}}{s^{\infty}_{14}}|\ne1.
\end{equation}
The first assumption of (\ref{AssumptionsCoef}) implies that $s^{\infty}_{44}\ne-1$. The second one is needed so that we can solve system (\ref{system}) with respect to $a(L)$ and $b(L)$ (again we use relations (\ref{relPart}) below to get this statement). The authors do not know how to proceed without the latter assumption.\\
\newline 
When (\ref{AssumptionsCoef}) is true, for all $L>1$ the denominators appearing in the definition of $a(L)$, $b(L)$ are not null. Then replacing in (\ref{ansatzTrapped}) $a(L)$, $b(L)$ by its expression derived above, we obtain $s_{22}(L)=s^{\mrm{asy}}_{22}(L)+\dots$ with 
\begin{equation}\label{DefSetComplex}
s^{\mrm{asy}}_{22}(L) = s^{\infty}_{22}+\cfrac{s^{\infty}_{24}s^{\infty}_{43}s^{\infty}_{32}-s^{\infty}_{23}(1+s^{\infty}_{44})s^{\infty}_{32}}{(1+s^{\infty}_{44})(-e^{-2ikL}+s^{\infty}_{33})-s^{\infty}_{34}s^{\infty}_{43}}+\cfrac{s^{\infty}_{24}(e^{-2ikL}-s^{\infty}_{33})s^{\infty}_{42}+s^{\infty}_{23}s^{\infty}_{34}s^{\infty}_{42}}{(1+s^{\infty}_{44})(-e^{-2ikL}+s^{\infty}_{33})-s^{\infty}_{34}s^{\infty}_{43}}.
\end{equation}
Below we prove the following result.
\begin{lemma}\label{LemmaCaracCircle}
Assume that the coefficients of the matrix $\mathbb{S}^{\infty}\in\Cplx_{4\times4}$ satisfy Assumptions (\ref{AssumptionsCoef}). Then $\{s^{\mrm{asy}}_{22}(L)\,|\,L\in(1;+\infty)\}$ is  the unit circle $\mathscr{C}:=\{z\in\Cplx\,|\,|z|=1\}$.
\end{lemma}
\noindent From the error estimate $|s_{22}(L)-s^{\mrm{asy}}_{22}(L)| \le C\,e^{-\sqrt{4\pi^2/\ell^2-k^2} L/2}$, the relation $|s_{22}(L)|=1$ for all $L>1$, and the fact that $L\mapsto s_{22}(L)$ is continuous, we deduce that we have 
$\{s_{22}(L)\,|\,L\in(1;+\infty)\}=\mathscr{C}$. And in particular, we infer that $L \mapsto s_{22}(L)$ passes through the point of affix $-1+i0$ almost periodically (the period is equal to $\pi/k$). From Lemma \ref{LemmaExistenceTrappedMode}, we deduce the following theorem, the main result of the section.
\begin{theorem}\label{MainThmPart2}
Assume that the coefficients of the matrix $\mathbb{S}^{\infty}\in\Cplx_{4\times4}$ satisfy Assumptions (\ref{AssumptionsCoef}). Then there is an infinite number of $L>1$ such that there are trapped modes for the Neumann Problem (\ref{PbChampTotalSym}) in $\om_L$. 
\end{theorem}
\begin{remark}
Note that the period of the function $L\mapsto s^{\mrm{asy}}_{22}(L)$ is equal to $\pi/k$ while the period of $L\mapsto\rD^{\mrm{asy}}(L)$ (see (\ref{defCoeffLim})) is equal to $2\pi/(k\sqrt{3})$. Thus in this geometry, trapped modes occur more frequently (with respect to $L\to+\infty$) than perfectly invisible modes. 
\end{remark}
\begin{remark}
If $u$ is a trapped mode for the Neumann Problem (\ref{PbChampTotalSym}) in $\om_L$, then the function $v$ such that $v=u$ in $\om_L$ and $v(x,y)=u(-x,y)$ in $\Om_L\setminus\overline{\om_L}$ is a trapped mode for the Neumann Problem (\ref{PbInitial}) in the unfold geometry $\Om_L$. 
\end{remark}

\subsection{Proof of Lemma \ref{LemmaCaracCircle}}
In this paragraph, we prove Lemma \ref{LemmaCaracCircle} which ensures that the set $\{s^{\mrm{asy}}_{22}(L)\,|\,L\in(1;+\infty)\}$ defined in (\ref{DefSetComplex}) coincides with the unit circle. \\
\newline
The already met function $u'=w^-_{1}+w^+_{1}=\sqrt{2/k}\cos(kx)$ is a solution of Problem (\ref{PbChampTotalSym}) set in the limit geometry $\om_{\infty}$ admitting the decomposition 
\[
u' = \chi_l\,(w^-_{1}+w^+_{1})+\cfrac{\sqrt{\ell}}{i\sqrt{2k}}\,\chi_t\,(w^-_{4}-w^+_{4})+\tilde{u}'.
\]
where $\tilde{u}'$ decays as $O(e^{\sqrt{4\pi^2-k^2}x})$ for $x\to-\infty$ and as $O(e^{-\sqrt{4\pi^2/\ell^2-k^2}y})$ for $y\to+\infty$. Set $\lambda=\sqrt{\ell}/(i\sqrt{2k})$. Observing that $u'-u^{\infty}_1+\lambda\,u^{\infty}_4$ has the same decay as $\tilde{u}'$, we deduce 
\begin{equation}\label{relPart}
s^{\infty}_{11}+\lambda\,s^{\infty}_{41}=1,\qquad s^{\infty}_{12}+\lambda\, s^{\infty}_{42}=0,
\qquad s^{\infty}_{13}+\lambda\,s^{\infty}_{43}=0,\qquad s^{\infty}_{14}+\lambda\,s^{\infty}_{44}=-\lambda.
\end{equation}
This allows one to write 
\[
\begin{array}{lcl}
s^{\mrm{asy}}_{22} &=& s^{\infty}_{22}+\cfrac{s^{\infty}_{24}s^{\infty}_{43}s^{\infty}_{32}-s^{\infty}_{23}(1+s^{\infty}_{44})s^{\infty}_{32}}{(1+s^{\infty}_{44})(-e^{-2ikL}+s^{\infty}_{33})-s^{\infty}_{34}s^{\infty}_{43}}+\cfrac{s^{\infty}_{24}(e^{-2ikL}-s^{\infty}_{33})s^{\infty}_{42}+s^{\infty}_{23}s^{\infty}_{34}s^{\infty}_{42}}{(1+s^{\infty}_{44})(-e^{-2ikL}+s^{\infty}_{33})-s^{\infty}_{34}s^{\infty}_{43}}\\[14pt]
 &= & s^{\infty}_{22}+\cfrac{2s^{\infty}_{24}s^{\infty}_{13}s^{\infty}_{32}-s^{\infty}_{23}s^{\infty}_{14}s^{\infty}_{32}-s^{\infty}_{12}(-e^{-2ikL}+s^{\infty}_{33})s^{\infty}_{42}}{s^{\infty}_{14}(-e^{-2ikL}+s^{\infty}_{33})-s^{\infty}_{13}s^{\infty}_{34}}\\[14pt]
        &= & s^{\infty}_{22}-\cfrac{s^{\infty}_{12}s^{\infty}_{42}}{s^{\infty}_{14}}+\cfrac{1}{s^{\infty}_{14}}\cfrac{2s^{\infty}_{24}s^{\infty}_{13}s^{\infty}_{32}-s^{\infty}_{23}s^{\infty}_{14}s^{\infty}_{32}-\cfrac{s^{\infty}_{12}s^{\infty}_{42}s^{\infty}_{13}s^{\infty}_{34}}{s^{\infty}_{14}}}{-e^{-2ikL}+s^{\infty}_{33}-\cfrac{s^{\infty}_{13}s^{\infty}_{34}}{s^{\infty}_{14}}}\ .
\end{array}
\]
Set 
\[
a=s^{\infty}_{33}-\cfrac{s^{\infty}_{13}s^{\infty}_{34}}{s^{\infty}_{14}},\quad b=\cfrac{1}{s^{\infty}_{14}}\,(\,2s^{\infty}_{24}s^{\infty}_{13}s^{\infty}_{32}-s^{\infty}_{23}s^{\infty}_{14}s^{\infty}_{32}-\cfrac{s^{\infty}_{12}s^{\infty}_{42}s^{\infty}_{13}s^{\infty}_{34}}{s^{\infty}_{14}}\,)\quad\mbox{ and }\quad c=s^{\infty}_{22}-\cfrac{s^{\infty}_{12}s^{\infty}_{42}}{s^{\infty}_{14}}.
\]
One can check that $b=-d^2$ with 
\[
d=s^{\infty}_{23}-\cfrac{s^{\infty}_{24}s^{\infty}_{13}}{s^{\infty}_{14}}.
\]
With this notation, we have $s^{\mrm{asy}}_{22}=c-d^2/(-e^{-2ikL}+a)$. Thus, as $L\to+\infty$, the coefficient $s^{\mrm{asy}}_{22}(L)$ runs on the set $\{c-d^2/(z+a)\,|\,z\in\mathscr{C}\}$. Working with the M\"{o}bius transform (same result as in the previous section), we deduce that $\{c-d^2/(z+a)\,|\,z\in\mathscr{C}\}$ coincides with the circle centered at 
\begin{equation}\label{caracCircle}
\alpha:=c+\cfrac{d^2\,\overline{a}}{1-|a|^2}\qquad\mbox{ of radius }\qquad\rho:=\cfrac{|d|^2}{1-|a|^2}.
\end{equation}
Therefore, to complete the proof of Lemma \ref{LemmaCaracCircle}, it remains to show that $\alpha=0$ and $\rho=1$.\\
\newline
$\star$ First we establish that $\rho=1$. This is equivalent to show that $|a|^2+|d|^2=1 \Leftrightarrow I=|s^{\infty}_{14}|^2$ with 
\begin{equation}\label{quantiteI}
I:=(|s^{\infty}_{33}|^2+|s^{\infty}_{23}|^2)|s^{\infty}_{14}|^2
+(|s^{\infty}_{24}|^2+|s^{\infty}_{34}|^2)|s^{\infty}_{13}|^2-2\,\Re e\,(s^{\infty}_{14}\overline{s^{\infty}_{13}}\,(s^{\infty}_{23}\overline{s^{\infty}_{24}}+s^{\infty}_{33}\overline{s^{\infty}_{34}})).
\end{equation}
Since $\mathbb{S}^{\infty}$ is unitary and symmetric, we have the identity
\begin{equation}\label{rel6}
s^{\infty}_{13}\overline{s^{\infty}_{14}}+s^{\infty}_{23}\overline{s^{\infty}_{24}}+s^{\infty}_{33}\overline{s^{\infty}_{34}}+s^{\infty}_{34}\overline{s^{\infty}_{44}}=0.
\end{equation}
Using (\ref{rel6}) in (\ref{quantiteI}), we get
\begin{equation}\label{quantiteIbis}
\begin{array}{lcl}
I&=&(|s^{\infty}_{23}|^2+|s^{\infty}_{33}|^2)|s^{\infty}_{14}|^2
+(|s^{\infty}_{24}|^2+|s^{\infty}_{34}|^2)|s^{\infty}_{13}|^2+2\,\Re e\,(s^{\infty}_{14}\overline{s^{\infty}_{13}}\,(s^{\infty}_{13}\overline{s^{\infty}_{14}}+s^{\infty}_{34}\overline{s^{\infty}_{44}}))\\[4pt]
&=&(|s^{\infty}_{13}|^2+|s^{\infty}_{23}|^2+|s^{\infty}_{33}|^2)|s^{\infty}_{14}|^2
+(|s^{\infty}_{14}|^2+|s^{\infty}_{24}|^2+|s^{\infty}_{34}|^2)|s^{\infty}_{13}|^2+2\,\Re e\,(s^{\infty}_{14}\overline{s^{\infty}_{13}} s^{\infty}_{34}\overline{s^{\infty}_{44}}).
\end{array}
\end{equation}
Using (\ref{relPart}), we can write
\begin{equation}\label{quantiteInterBis}
\begin{array}{lcl}
2\,\Re e\,(s^{\infty}_{14}\overline{s^{\infty}_{13}}s^{\infty}_{34}\overline{s^{\infty}_{44}})=-2\,|s^{\infty}_{34}|^2|\Re e\,(s^{\infty}_{14}\,\overline{\lambda}\,\overline{s^{\infty}_{44}}) &=&-|s^{\infty}_{34}|^2(|\lambda|^2-|s^{\infty}_{34}|^2-|\lambda|^2|s^{\infty}_{44}|^2)\\[5pt]
&=&-|s^{\infty}_{13}|^2+|s^{\infty}_{34}|^2|s^{\infty}_{14}|^2+|s^{\infty}_{13}|^2|s^{\infty}_{44}|^2.
\end{array}
\end{equation}
Plugging (\ref{quantiteInterBis}) in (\ref{quantiteIbis}), we obtain
\begin{equation}\label{conclusionRho}
I=(|s^{\infty}_{13}|^2+|s^{\infty}_{23}|^2+|s^{\infty}_{33}|^2+|s^{\infty}_{34}|^2)|s^{\infty}_{14}|^2
+(|s^{\infty}_{14}|^2+|s^{\infty}_{24}|^2+|s^{\infty}_{34}|^2+|s^{\infty}_{44}|^2-1)|s^{\infty}_{13}|^2=|s^{\infty}_{14}|^2.
\end{equation}
To derive the second equality in (\ref{conclusionRho}), we used again the fact $\mathbb{S}^{\infty}$ is unitary. Identity (\ref{conclusionRho}) ensures that $\rho=1$.\\
\newline
$\star$ Now, we prove that $\alpha=c+d^2\overline{a}/(1-|a|^2)=0$. Since $\rho=|a|^2+|d|^2=1$, this is equivalent to show that $c\overline{d}+d\overline{a}=0$. We have 
\begin{equation}\label{TermToAssess}
\begin{array}{lcl}
|s^{\infty}_{14}|^2(c\overline{d}+d\overline{a}) &=&\phantom{-}|s^{\infty}_{14}|^2(s^{\infty}_{22}\overline{s^{\infty}_{23}}+s^{\infty}_{23}\overline{s^{\infty}_{33}})+|s^{\infty}_{24}|^2s^{\infty}_{12}\overline{s^{\infty}_{13}}+|s^{\infty}_{13}|^2s^{\infty}_{24}\overline{s^{\infty}_{34}}\\[4pt]
 & & -s^{\infty}_{14}\overline{s^{\infty}_{13}}(s^{\infty}_{23}\overline{s^{\infty}_{34}}+s^{\infty}_{22}\overline{s^{\infty}_{24}})
 -s^{\infty}_{24}\overline{s^{\infty}_{14}}(s^{\infty}_{12}\overline{s^{\infty}_{23}}+s^{\infty}_{13}\overline{s^{\infty}_{33}}).
\end{array}
\end{equation}
Since $\mathbb{S}^{\infty}$ is unitary and symmetric, we have
\[
\begin{array}{lcl}
s^{\infty}_{12}\overline{s^{\infty}_{13}}+s^{\infty}_{22}\overline{s^{\infty}_{23}}+s^{\infty}_{23}\overline{s^{\infty}_{33}}+s^{\infty}_{24}\overline{s^{\infty}_{34}}&=&0\\[4pt]
s^{\infty}_{12}\overline{s^{\infty}_{14}}+s^{\infty}_{22}\overline{s^{\infty}_{24}}+s^{\infty}_{23}\overline{s^{\infty}_{34}}+s^{\infty}_{24}\overline{s^{\infty}_{44}}&=&0\\[4pt]
s^{\infty}_{11}\overline{s^{\infty}_{31}}+s^{\infty}_{12}\overline{s^{\infty}_{23}}+s^{\infty}_{13}\overline{s^{\infty}_{33}}+s^{\infty}_{14}\overline{s^{\infty}_{34}}&=&0.
\end{array}
\]
Using these three identities in (\ref{TermToAssess}), we find 
\begin{equation}\label{lastEquality}
\begin{array}{lcl}
|s^{\infty}_{14}|^2(c\overline{d}+d\overline{a}) &=&-|s^{\infty}_{14}|^2(s^{\infty}_{12}\overline{s^{\infty}_{13}}+s^{\infty}_{24}\overline{s^{\infty}_{34}})+|s^{\infty}_{24}|^2s^{\infty}_{12}\overline{s^{\infty}_{13}}+|s^{\infty}_{13}|^2s^{\infty}_{24}\overline{s^{\infty}_{34}}\\[4pt]
 & & +s^{\infty}_{24}\overline{s^{\infty}_{13}}(s^{\infty}_{11}\overline{s^{\infty}_{14}}+s^{\infty}_{14}\overline{s^{\infty}_{44}})
+|s^{\infty}_{14}|^2 (s^{\infty}_{12}\overline{s^{\infty}_{13}}+ s^{\infty}_{24}\overline{s^{\infty}_{34}})\\[8pt]
& =& |s^{\infty}_{24}|^2s^{\infty}_{12}\overline{s^{\infty}_{13}}+|s^{\infty}_{13}|^2s^{\infty}_{24}\overline{s^{\infty}_{34}}-|s^{\infty}_{24}|^2s^{\infty}_{12}\overline{s^{\infty}_{13}}+|s^{\infty}_{13}|^2s^{\infty}_{24}\overline{s^{\infty}_{34}}=0.
\end{array}
\end{equation}
To derive the last line of the above equality, we used the relation
\[
s^{\infty}_{11}\overline{s^{\infty}_{14}}+s^{\infty}_{12}\overline{s^{\infty}_{24}}+s^{\infty}_{13}\overline{s^{\infty}_{34}}+s^{\infty}_{14}\overline{s^{\infty}_{44}}=0.
\]
This gives $\alpha=0$ and allows us to conclude that the set $\{s^{\mrm{asy}}_{22}(L)\,|\,L\in(1;+\infty)\}$ is indeed the unit circle.
\begin{remark}
One can note that the special form of the matrix $\mathbb{S}^{\infty}$ (see relations (\ref{relPart})) due to the particular choice of the geometry is used only in the proof of $\rho=1$. Once this property is established, the fact that $\mathbb{S}^{\infty}$ is unitary suffices to conclude that $\alpha=0$.
\end{remark}

\section{Numerical experiments}\label{SectionNumExpe}

\subsection{Perfectly invisible modes}\label{paragraphPerfectRef}
In the first series of experiments, we exhibit some $L>1$ such that perfect invisibility holds in the geometry $\Om_L$ defined in (\ref{defOriginalDomain}). For each $L$ in a given range, we compute numerically the transmission coefficient $\tcoef$ defined in (\ref{DefScatteringCoeff}). To proceed, we use a $\mrm{P}2$ finite element method in a truncated waveguide. On the artificial boundary created by the truncation, a Dirichlet-to-Neumann operator with $\mrm{15}$ terms serves as a transparent condition. We take $k=0.8\pi$ so that the width of the vertical branch of $\Om_L$ is equal to $2\ell=2\pi/k=2.5$. In Figure \ref{figResult3} left, we display the curve $L\mapsto \tcoef(L)$ for $L\in(1.3;8)$. In accordance with the results obtained in \S\ref{paragraphConcluPart1}, we observe that when $L\to+\infty$, $L\mapsto \tcoef(L)$ runs on the circle of radius $1/2$ centered at $1/2+0i$ in the complex plane. In particular, $L\mapsto \tcoef(L)$ passes through the point of affix $1+0i$ (Theorem \ref{MainThmPart1}). In Figure \ref{figResult3} right, we display the curve $L\mapsto -\ln |\tcoef(L)-1|$ for $L\in(1.3;8)$. The picks correspond to the values of $L$ such that $\tcoef(L)=1$. According to the proof of Theorem \ref{MainThmPart1}, we expect that the picks are almost periodic with a distance between two picks tending to $2\pi/(k\sqrt{3})=2.5/\sqrt{3}\approx1.44$ as $L\to+\infty$. The numerical results we get are coherent with this value. In Figure \ref{figResultT1}, we represent the real part of the total field $v$ defined in (\ref{DefScatteringCoeff}) as well as $v-w^-_1$ for $L=2.5756$ (first pick of Figure \ref{figResult3} right). Finally in Figure \ref{figResultT1Gros}, we display $v$ and $v-w^-_1$ in another geometry where $\tcoef(L)=1$. Here the domain is 
\begin{equation}\label{defDomainBis}
\tilde{\Om}_L=\{(x,y)\in\R^2\,|\,0<y<g_L(x)\}
\end{equation}
where $g_L:\R\to\R$ is the even staircase function such that 
$g_L(x)=L$ for $ 0 \le x <\ell$, $g_L(x)=2.5$ for $\ell < x <2\ell$, $g_L(x)=2$ for $2\ell < x <3\ell$, $g_L(x)=1.5$ for $3\ell < x <4\ell$ and 
$g_L(x)=1$ for $4\ell<x$ (see Figure \ref{figResultT1Gros}). The two important points are to choose $g_L$ so that $u'=w^-_1+w^+_1=\sqrt{2/k}\cos(kx)$ satisfies the initial problem (\ref{PbInitial}) set in $\tilde{\Om}_L$ for all $L>1$ and to preserve the symmetry of the geometry with respect to the $(Oy)$ axis. Then playing with $L$ as explained in Section \ref{SectionPerfectInvisibility}, one can get $\tcoef(L)=1$ (theoretically and numerically). Of course other staircase geometries satisfy the two mentioned properties.

\begin{remark}\label{ExplanationRealPart}
In the numerical experiments leading to Figures \ref{figResultT1}, \ref{figResultT1Gros}, one finds that $\Re e\,(v-w^-_1)\equiv 0$. This can be proved. Indeed, observing that $v-w^-_1+\overline{v-w^-_1}=v+\overline{v}-(w^+_1+w^-_1)$, we deduce that $v-w^-_1+\overline{v-w^-_1}$ is a solution of Problem (\ref{PbInitial}). Since $v-w^-_1+\overline{v-w^-_1}$ is exponentially decaying as $|x|\to+\infty$ (because $\tcoef=1$), we infer that if trapped modes do not exist, there holds $v-w^-_1+\overline{v-w^-_1}\equiv0\Leftrightarrow \Re e\,(v-w^-_1)\equiv 0$. Note that in the proof, we use again the fact that in the particular geometries considered in the present work, $w^+_1+w^-_1$ is a solution of Problem (\ref{PbInitial}).
\end{remark}

\begin{figure}[!ht]
\centering
\includegraphics[width=0.48\textwidth]{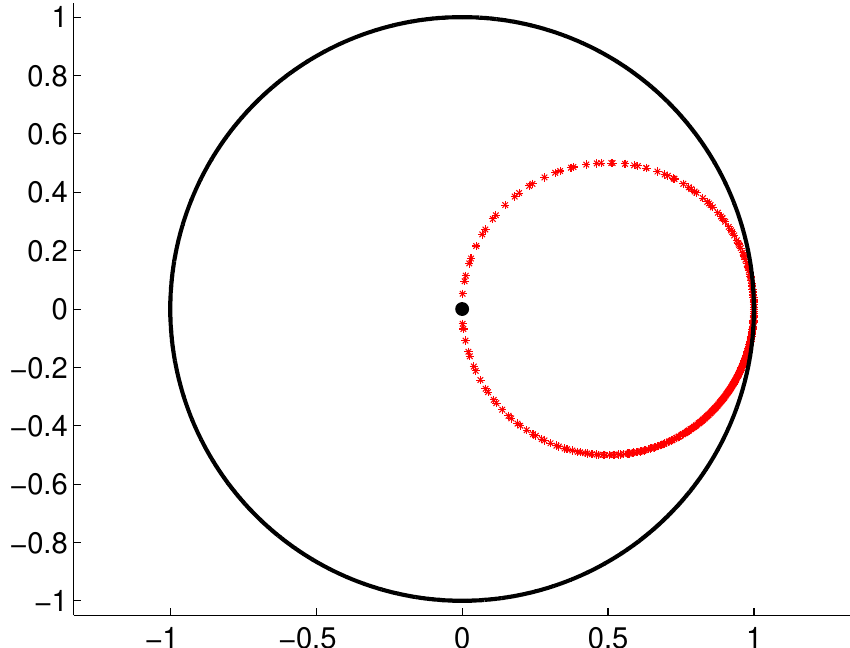}\quad\includegraphics[width=0.48\textwidth]{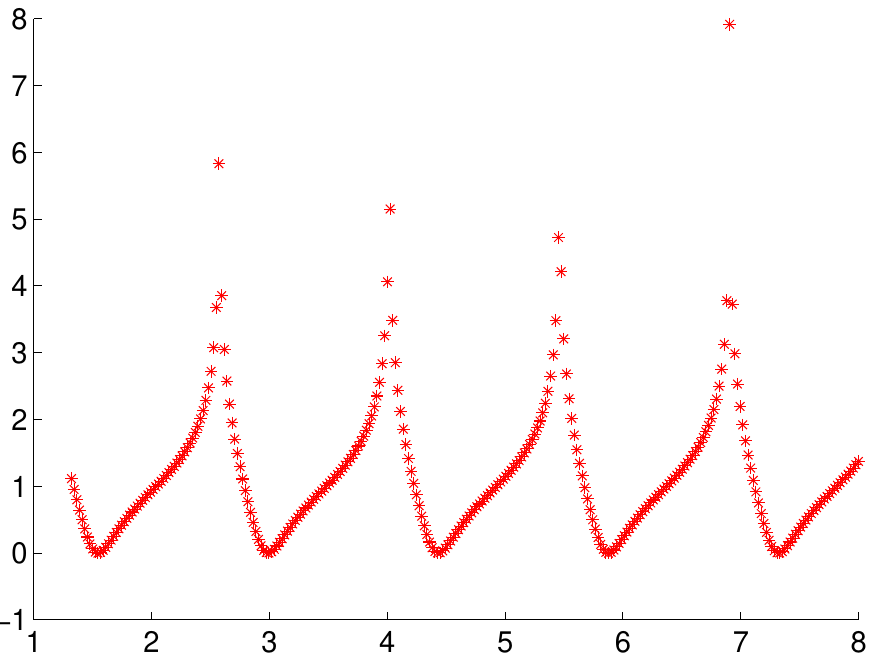}
\caption{Left: coefficient $L\mapsto \tcoef(L)$ for $L\in [1.3;8]$. According to the conservation of energy (\ref{ConservationOfNRJ}), we know that the scattering coefficient $\tcoef$ is located inside the unit disk delimited by the black bold line. Right: curve $L\mapsto -\ln|\tcoef(L)-1|$ for $L\in [1.3;8]$.\label{figResult3}} 
\end{figure}

\begin{figure}[!ht]
\centering
\includegraphics[width=0.95\textwidth]{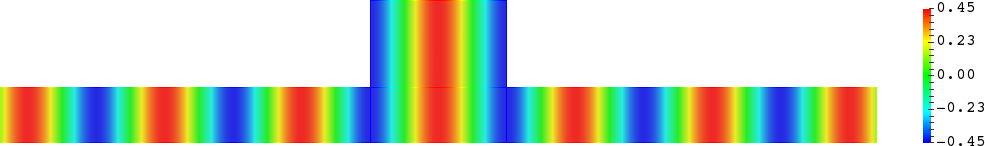}\\[10pt]
\includegraphics[width=0.95\textwidth]{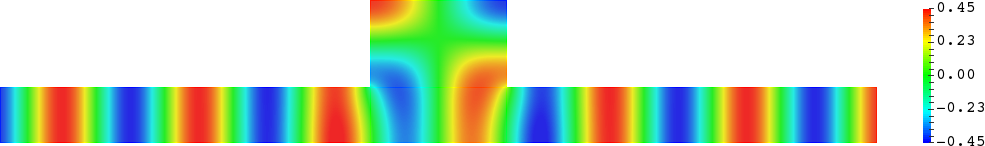}\\[10pt]
\includegraphics[width=0.95\textwidth]{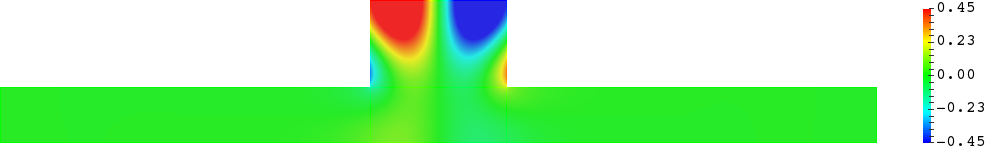}
\caption{$\Re e\,v$ (top), $\Im m\,v$ (middle) and $\Im m\,(v-w^-_1)$ (bottom) for $L=2.5756$ where $v$ is the function introduced in (\ref{DefScatteringCoeff}). One finds that $\Re e\,(v-w^-_1)\equiv 0$. The latter result is specific to the geometry considered here (see Remark \ref{ExplanationRealPart}).
\label{figResultT1}}
\end{figure}

\begin{figure}[!ht]
\centering
\includegraphics[width=\textwidth]{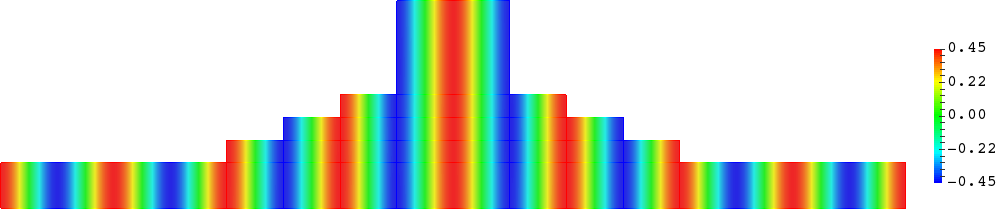}\\[15pt]
\includegraphics[width=\textwidth]{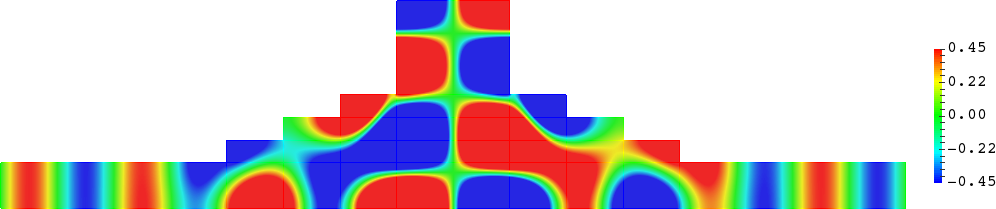}\\[15pt]
\includegraphics[width=\textwidth]{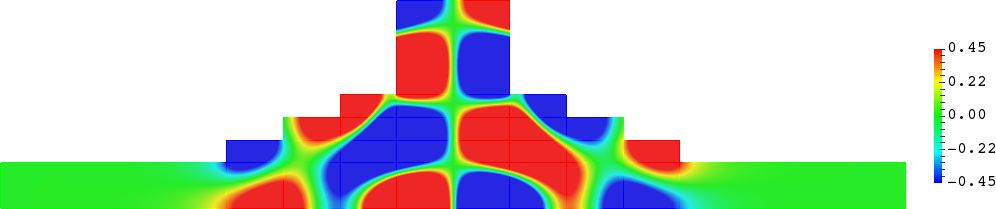}
\caption{$\Re e\,v$ (top), $\Im m\,v$ (middle) and $\Im m\,(v-w^-_1)$ (bottom) in the waveguide $\tilde{\Om}_L$ defined in (\ref{defDomainBis})  for $L=4.5808$. One finds that $\Re e\,(v-w^-_1)\equiv 0$. The latter result is specific to the geometry considered here (see Remark \ref{ExplanationRealPart}).
\label{figResultT1Gros}}
\end{figure}
\newpage

\subsection{Perfect reflection}

As indicated in Remark \ref{RmkMirorEffect}, the method proposed in Section \ref{SectionPerfectInvisibility} allows one also to exhibit situations where $\tcoef(L)=0$ (perfect reflection) and $\rcoef(L)=1$. In Figure \ref{MirrorEffect}, we display this mirror effect, with the energy completely backscattered, in the geometry $\tilde{\Om}_L$ defined in (\ref{defDomainBis}) for a well chosen $L$. More precisely, we display the real part of the total field and we observe that it is indeed exponentially decaying as $x\to+\infty$.

\begin{remark}\label{ExplanationTzero}
In Figure \ref{MirrorEffect}, we observe that $\Im m\,v\equiv 0$ when $\rcoef=1$. This is a general result that holds without assumption on the geometry of the waveguide as soon as $\rcoef=1$. Indeed, if $\tcoef=0$ then $|\rcoef|=1$ and one can check that $v-\overline{\rcoef}\,\overline{v}$ is a solution of Problem (\ref{PbInitial}) which is exponentially decaying as $|x|\to+\infty$. Therefore, if trapped modes do not exist, we deduce that $v-\overline{\rcoef}\,\overline{v}\equiv0\Leftrightarrow v=\overline{\rcoef}\,\overline{v}$. In particular, if $\rcoef=1$, we obtain $\Im m\,v\equiv 0$.
\end{remark}

\begin{figure}[!ht]
\centering
\includegraphics[width=\textwidth]{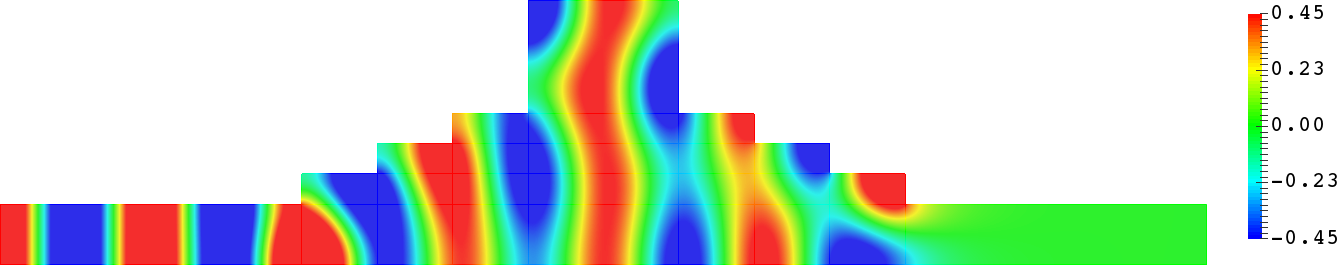}
\caption{$\Re e\,v$ in the waveguide $\tilde{\Om}_L$ defined in (\ref{defDomainBis}) for $L=4.3758$. One finds that $\Im m\,v\equiv 0$. The latter result actually holds in any waveguide where $\rcoef=1$ (see Remark \ref{ExplanationTzero}).
\label{MirrorEffect}}
\end{figure}

\subsection{Trapped modes}
In the third series of experiments, we give examples of geometries supporting trapped modes for the initial Helmholtz problem  (\ref{PbInitial}) with Neumann boundary condition. For each $L$ in a given range, this time we compute numerically the  coefficients of the augmented scattering matrix $\mathbb{S}\in\Cplx_{2\times2}$ defined in (\ref{AugmentedScatteringDef}). Again we use a $\mrm{P}2$ finite element method set in a truncated waveguide. We emphasize here that we need to work with a well-suited Dirichlet-to-Neumann map to deal with the wave packet $W_2^+$ appearing in the decompositions of $u_1$, $u_2$ in (\ref{defu1u2}). In Figure \ref{AugmentedScatteringMatrix} left, we display the coefficients  $L\mapsto s_{ij}(L)$ for $i,j\in\{1,2\}$ and $L\in(1.3;8)$. The wavenumber is set to $k=0.8\pi$. As shown by the theory, indeed we have $s_{11}(L)=1$ and $s_{12}(L)=s_{21}(L)=0$ for all $L>1$. Moreover, we observe that $L\mapsto s_{22}(L)$ runs on the unit circle. Consequently, it indeed goes through the point of affix $-1+i0$ which guarantees the existence of trapped modes for certain $L>1$ (Theorem \ref{MainThmPart2}). In Figure \ref{AugmentedScatteringMatrix} right, we display the curve $L\mapsto -\ln |s_{22}(L)+1|$ for $L\in(1.3;8)$. Indeed, it has some picks indicating values of $L$ such that $s_{22}(L)=-1$. According to the proof of Theorem \ref{MainThmPart2}, we expect a distance between two picks approximately equal to $\pi/k=1.25$. The numerical results are in good agreement with this value. In Figure \ref{TrappedMode}, we display the real part of a trapped mode in $\om_L$ for $L=2.5524$ (first pick of Figure \ref{AugmentedScatteringMatrix} right). Finally, in Figure \ref{TrappedModeFancy}, we display the real part of a trapped mode in the waveguide $\tilde{\Om}_L$ defined in (\ref{defDomainBis}) for $L=3.8273$. This trapped mode has been obtained symmetrising the trapped mode of the half-waveguide problem with Neumann boundary conditions with respect to the $(Oy)$ axis. Following the analysis of Section \ref{SectionExistenceOfTrappedModes}, we can construct trapped modes in any staircase geometry satisfying the two properties mentioned at the end of \S\ref{paragraphPerfectRef}.

\begin{figure}[!ht]
\centering
\includegraphics[width=0.48\textwidth]{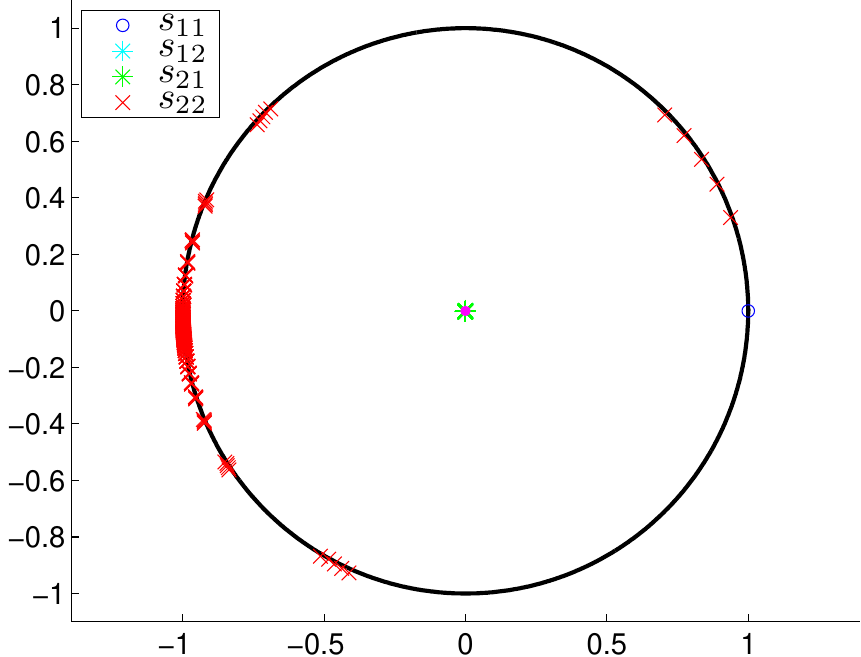}\quad\includegraphics[width=0.48\textwidth]{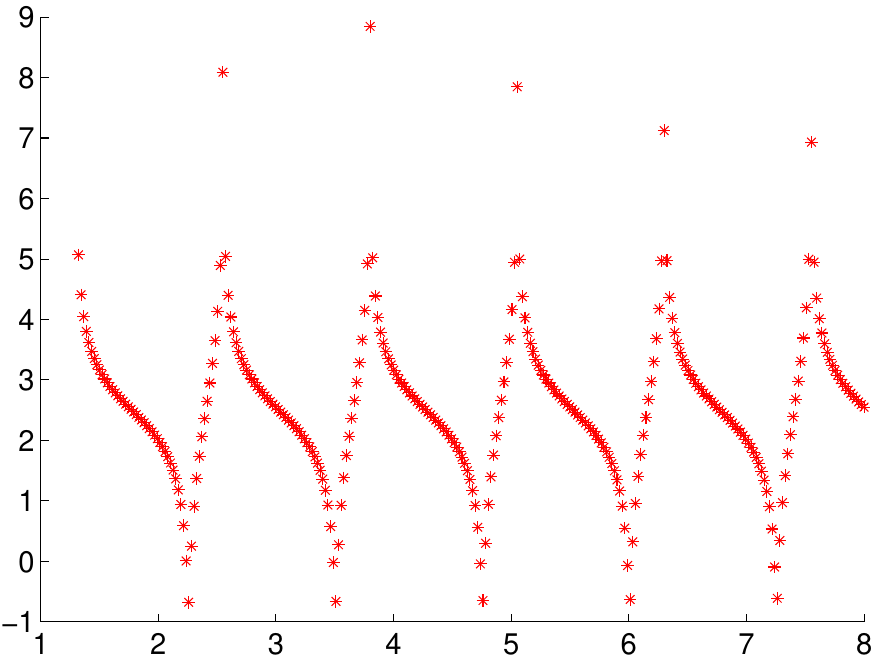}
\caption{Left: coefficients of the augmented scattering matrix defined in (\ref{AugmentedScatteringDef}) for $L\in [1.3;8]$. Right: curve $L\mapsto -\ln|s_{22}(L)+1|$ for $L\in [1.3;8]$.\label{AugmentedScatteringMatrix}} 
\end{figure}

\begin{figure}[!ht]
\centering
\includegraphics[scale=0.55]{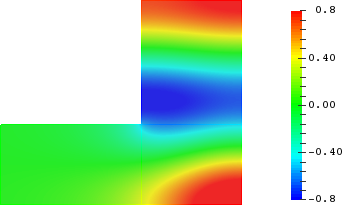}
\caption{Real part of a trapped mode in the geometry $\om_L$ defined in (\ref{defHalfWaveguide}) for $L=2.5524$. \label{TrappedMode}}
\end{figure}

\begin{figure}[!ht]
\centering
\includegraphics[scale=0.48]{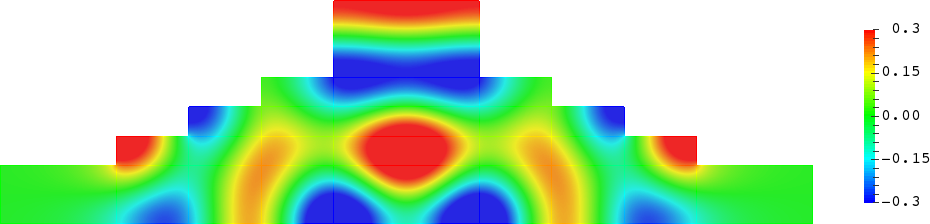}
\caption{Real part of a trapped mode in the waveguide $\tilde{\Om}_L$ defined in (\ref{defDomainBis})  for $L=3.8273$. \label{TrappedModeFancy}}
\end{figure}

\newpage

\section{Conclusion}\label{SectionConclusion}

In this article, we proved the existence of perfectly invisible and trapped modes in simple particular geometries for the Helmholtz problem with Neumann boundary conditions. Importantly, the waveguide has to be symmetric with respect to the $(Oy)$ axis and endowed with a branch of tunable height whose width coincides with the wavelength of the incident wave. The analysis has been done in the domain introduced in (\ref{defOriginalDomain}). Other examples of geometrical situations where the method can be performed have been presented in Section \ref{SectionNumExpe}. Our technique seems specific to the problem with Neumann boundary conditions and it does not look simple to modify it to consider for example the case of Dirichlet boundary conditions (appearing e.g. in the analysis of quantum waveguides). Finally, what we did in $\mrm{2D}$ can be adapted to higher dimension. 
 
\section*{Annex}
In this Annex, for the convenience of the reader, we provide the detail of the proof of a well-known proposition.
\begin{proposition}\label{PropositionUnitary}
The scattering matrix $\mathcal{S}^{\infty}$ defined in (\ref{UnboundedScatteringMatrix}) is unitary and symmetric. 
\end{proposition}
\begin{proof}
Define the symplectic (sesquilinear and anti-hermitian ($q(\varphi,\psi)=-\overline{q(\psi,\varphi)}$)) form $q(\cdot,\cdot)$ such that for all $\varphi,\psi\in\mH^1_{\loc}(\om_\infty)$
\[
q(\varphi,\psi)=\int_{\Sigma} \cfrac{\partial \varphi}{\partial n}\overline{\psi}-\varphi\cfrac{\partial \overline{\psi}}{\partial n}\,d\sigma.
\]
Here $\Sigma:=\{-2\ell\}\times(0;1)\cup (-\ell;0)\times\{1+2\tau\}$, $\partial_n=-\partial_x$ on $\{-2\ell\}\times(0;1)$, $\partial_n=\partial_y$ on $(-\ell;0)\times\{1+2\tau\}$. Moreover, $\mH^1_{\loc}(\om_\infty)$ refers to the Sobolev space of functions $\varphi$ such that $\varphi|_{\mathcal{O}}\in\mH^1(\mathcal{O})$ for all bounded domains $\mathcal{O}\subset\om_\infty$. Integrating by parts and using that the functions $U^{\infty}_1$, $U^{\infty}_2$ satisfy the Helmholtz equation, we obtain $q(U^{\infty}_i,U^{\infty}_j)=0$ for $i,\,j\in\{1,2\}$. On the other hand, decomposing $U^{\infty}_1$, $U^{\infty}_2$ in Fourier series on $\Sigma$, we find 
\[
\begin{array}{c}
q(U^{\infty}_1,U^{\infty}_1) = (-1+|S^{\infty}_{11}|^2+|S^{\infty}_{12}|^2)\,i,\quad q(U^{\infty}_2,U^{\infty}_2) = (-1+|S^{\infty}_{22}|^2+|S^{\infty}_{21}|^2)\,i\\[5pt]
q(U^{\infty}_1,U^{\infty}_2)=-\overline{q(U^{\infty}_2,U^{\infty}_1)}=S^{\infty}_{11}\overline{S^{\infty}_{21}}+S^{\infty}_{12}\overline{S^{\infty}_{22}}.
\end{array}
\]
These relations allow us to prove that $\mathcal{S}^{\infty}\,\overline{\mathcal{S}^{\infty}}^{\top}=\mrm{Id}_{2\times 2}$, that is to conclude that $\mathcal{S}^{\infty}$ is unitary. On the other hand, one finds $q(U^{\infty}_1,\overline{U^{\infty}_2})=0=-S^{\infty}_{21}+S^{\infty}_{12}$. We deduce that $\mathcal{S}^{\infty}$ is symmetric.
\end{proof}

\bibliography{Bibli}
\bibliographystyle{plain}
\end{document}